\documentclass{amsart}
\newtheorem{thm}{Theorem}[section]

\newtheorem{prop}[thm]{Proposition}
\theoremstyle{definition}
\newtheorem{defn}[thm]{Definition}
\newtheorem{example}[thm]{Example}
\theoremstyle{remark}
\newtheorem{rem}[thm]{Remark}
\numberwithin{equation}{section}
\usepackage{graphicx}
\usepackage{amsmath,amsfonts,amssymb}
\usepackage{enumerate}
\begin{document}

\title[Singularities of discrete indefinite affine minimal surfaces]{Singularities of discrete indefinite affine minimal surfaces}%
\author{Marcos Craizer}% 
\email{craizer@puc-rio.br \newline
Marcos Craizer is a professor in Pontifical Catholic University of Rio de Janeiro.}%

\thanks{The author wants to thank CNPq for financial support during the preparation of this manuscript.}%

%\date{}%
%\dedicatory{}%
%\commby{}%
% ----------------------------------------------------------------
\begin{abstract}
A smooth affine minimal surface with indefinite metric can be obtained from a pair of smooth non-intersecting spatial curves by Lelieuvre's formulas. These surfaces may present singularities, which are gene\-rically cuspidal edges and swallowtails. By discretizing the initial curves, one can obtain by the discrete Lelieuvre's formulas a discrete affine minimal surface with inde\-finite metric. The aim of this paper is to define the singular edges and vertices of the corresponding discrete asymptotic net in such a way that the most relevant properties of the singular set of the smooth version remain valid.
\end{abstract}
\maketitle
% ---------------------------------------------------------------

\section{Introduction}

Discretization of a class of surfaces keeping some of the main properties
of the smooth version is the basic philosophy of Discrete Differential Geometry (\cite{Bobenko2008}). The discrete versions have the advantage of being easily represented in computers and can provide robust numerical methods for solving differential equations associated to the class.

Discrete asymptotic nets are natural nets for the discretization
of surfaces parameterized by asymptotic coordinates and have been the object of
recent and ancient research by many geometers, as one can see in the list of references of
this paper (\cite{Bobenko2008}, \cite{Craizer2010}, \cite{Rorig2014}, \cite{Kaferbock2013}). 
On the other hand, singularities of discrete nets is a recent topic of research (\cite{Rossman2018}). 

A smooth affine minimal surface with indefinite metric (IAMS) can be obtained from a pair of smooth non-intersecting curves in 
$\mathbb{R}^3$, $\alpha(u)$, $u\in I\subset \mathbb{R}$ and $\beta(v)$, $v\in J\subset \mathbb{R}$, by considering the co-normal vector field $\nu(u,v)=\alpha(u)-\beta(v)$ and Lelieuvre's formulas $f_u=\nu\times \nu_u$ and $f_v=-\nu\times\nu_v$.
The $(u,v)$ coordinates are then asymptotic and the Blaschke metric becomes $\Omega dudv$, where 
$\Omega=\left[\nu,\nu_u,\nu_v \right]$ (\cite{Nomizu1994}).
Such a surface present singularities at points $(u,v)$ such that $\Omega(u,v)=0$.
Generically, the singular set is a smooth curve and the admissible singularities are cuspidal edges and swallowtails.

By sampling the curves $\alpha$ e $\beta$, one obtain polygonal lines in $\mathbb{R}^3$. We shall keep calling these polygonal lines $\alpha$ and $\beta$. Defining discrete co-normals $\nu(u,v)=\alpha(u)-\beta(v)$, we can ontain a discrete asymptotic net by using the discrete Lelieuvre's formulas. Such a discrete asymptotic net is called a discrete affine minimal surface with indefinite metric (DIAMS). It can be shown that this discretization keeps some of the main properties of an IAMS (\cite{Craizer2010}).

In this paper, generalizing  the case of improper affine spheres (\cite{Craizer2023}), we propose definitions for the singularities of the DIAMS that also preserve the main properties of the generic singularities of an IAMS. These definitions are quite consistent, which is a remarkable fact, since it is not easy to have good definitions of singularities of discrete nets. For a discussion of this question, see \cite{Rossman2018}. For other classes of discrete surfaces with singularities, see \cite{Hoffmann2012} and \cite{Yasumoto2015}.

The discrete metric $\Omega$ can be defined at each quadrangle of the net. It is natural to define the singular edges as those whose metric at adjacent quadrangles changes the sign. We show that, under quite natural hypothesis, the singular set is a polygonal line, which can be thought as a discrete counterpart of a smooth curve. We also give strong arguments to, among the singular vertices, distinguish those who should be considered as cuspidal edges from those who should be considered swallowtails. Finally we provide some examples of DIAMS with bilinear interpolation, which visually reinforce the validity of our definitions.

The paper is organized as follows: In Section 2 we consider smooth affine minimal surfaces with indefinite metric, emphasizing the geometric conditions for singular points, and, among them, the swallowtail points. In section 3 we review the definitions concerning DIAMS without singularities. Section 4 deals with the singular edges and vertices of a DIAMS, and contains the main definitions and results of the paper.

\section{Smooth indefinite affine minimal surfaces with singularities}

\subsection{Indefinite affine minimal surfaces}

Consider a smooth immersion $f:U\subset\mathbb{R}^2\to\mathbb{R}^3$, where $U$ is an open subset of the plane. For 
$(u,v)\in U$, let
$$
L(u,v)=[f_u,f_v,f_{uu}], \ M(u,v)=[f_u,f_v,f_{uv}], \ N(u,v)=[f_u,f_v,f_{vv}],
$$
where $f_u (f_v)$ denotes the partial derivative of a function $f$ with respect to $u (v)$, and $[\cdot,\cdot,\cdot]$ denotes the determinant.
The surface $f$ is said to be non-degenerate if $LN-M^2\neq0$, and in this case, the Berwald-Blaschke metric is defined by
$ds^2=\frac{1}{|LN-M^2|^{1/4}}\left(L\,du^2+2M\,du\,dv+N\,dv^2\right).$ If $LN-M^2<0$, the metric is called \emph{indefinite} and the surface is locally hyperbolic, \emph{i.e.}, the tangent plane crosses the surface. In this paper
we shall assume that the affine surface has indefinite metric.
We may also assume, and we shall do it, that $(u,v)$ are \emph{asymptotic} coordinates, i.e.,  $L=N=0$. In this case, it is possible to take $M>0$, and the affine Blaschke metric takes the form $ds^2=2\Omega\,du\,dv$, where $\Omega^2=M$. The co-normal vector field $\eta$
is defined by
$$
\nu(u,v)=\frac{1}{\Omega}\left( f_u\times f_v\right).
$$
For more details, see \cite{Buchin1983}.

An immersion as above is an indefinite affine minimal surface (IAMS) if and only if 
$\nu_{uv}=0$, 
where $(u,v)$ are asymptotic coordinates (\cite{Nomizu1994}). Thus we can write 
\begin{equation}\label{Translation}
\nu(u,v)=\alpha(u)-\beta(v),
\end{equation}
for certain smooth non-intersecting curves $\alpha:I\to\mathbb{R}^3$ and $\beta:J\to\mathbb{R}^3$. 
Starting from the affine Blaschke co-normal map, one can recover the immersion using the smooth Leliuvre's formulas
\begin{equation*}\label{Lelieuvre}
f_u=\nu\times\nu_u,\ \ \ f_v=-\nu\times\nu_v,
\end{equation*}
where $\times$ denote the vector product in $\mathbb{R}^3$.
%These formulas are sometimes called the {\it affine Weierstrass representation formula} for affine minimal surfaces.
Using Equation \eqref{Translation}, we obtain that 
%the Blaschke metric is given by $\Omega dudv$, where $\Omega^2=M$, 
%\begin{equation}\label{eq:SmoothM}
%M(u,v)=\left[f_u,f_v,f_{uv}\right]. 
%\end{equation}
%One can verify that 
$\Omega(u,v)=\left[\nu,\nu_u,\nu_v\right]$ and so
\begin{equation}\label{eq:SmoothOmega}
\Omega(u,v)=\left[\alpha(u)-\beta(v),\alpha'(u),\beta'(v)\right]. 
\end{equation}
For a smooth surface without singularities, the sign of $\Omega$ does not change.

\subsection{Singularities}

Even if $\Omega$ changes sign, we shall call the surface obtained from Equation \eqref{Translation} an IAMS. In this context,  
$(u,v)\in I\times J$ is a singular point of the IAMS if and only if $\Omega(u,v)=0$. They correspond to the points where $f_u$ and $f_v$ are not linearly independent. 

For $(u_0,v_0)\in U$, fix a plane $\pi=\pi(u_0,v_0)$ transversal to $\alpha(u_0)-\beta(v_0)$ and denote by $P=P(u_0,v_0)$ the projection on $\pi$ parallel to
$\alpha(u_0)-\beta(v_0)$. 

\begin{prop}
The point $(u_0,v_0)\in U$ is singular if and only if $P\alpha'(u_0)$, $P\beta'(v_0)$ are parallel.
\end{prop}
\begin{proof}
A point $(u_0,v_0)\in U$ is singular if and only if 
$$
[\alpha(u_0)-\beta(v_0),P\alpha'(u_0),P\beta'(v_0)]=0,
$$
which is equivalent
to $P\alpha'(u_0)\times P\beta'(v_0)=0$. 
\end{proof}

It is easy to see that the singular set is a regular curve except at points 
$(u,v)\in I\times J$ such that
\begin{equation}\label{eq:NonRegularSingularCurve}
P\alpha'(u)\times P\beta'(v)=P\alpha''(u)\times P\beta'(v)=P\alpha'(u)\times P\beta''(v)=0.
\end{equation}
This is a co-dimension $3$ condition in the space of $2$-jets of $(\alpha(u),\beta(u))$. Thus, for generic
$(\alpha,\beta)\in C^{\infty}(I,\mathbb{R}^3)\times C^{\infty}(J,\mathbb{R}^3)$, condition \eqref{eq:NonRegularSingularCurve} do not occur. We conclude that the singular set is generically a smooth regular curve.

At a point $(u_0,v_0)\in S$, denote by $\eta=\eta(u_0,v_0)$ a vector in the kernel of $df(u_0,v_0)$. If $f_u+\lambda f_v=0$, then we can write $\eta=(1,\lambda)$. We shall distinguish the cases $\eta$ transversal to $S$ and $\eta$ tangent to $S$.
Next proposition says that generically, for most points of $S$, a singular point of an affine minimal surface with indefinite metric is a cuspidal edge, but, for some isolated points, it is a swallowtail.

\begin{prop}\label{prop:SmoothSingularities}(\cite{Kokubu2005}) Let $(u_0,v_0)\in S$ and $\eta=\eta(u_0,v_0)$ a vector in the kernel of $df(u_0,v_0)$. Then
\begin{enumerate}
\item  
If $\eta$ is transversal to $S$, the singularity is a cuspidal edge.

\item 
If $\eta$ is tangent to $S$, the singularity is generically a swallowtail.
\end{enumerate}
\end{prop}

\subsection{Geometric conditions for singularities}

We now describe some geometric conditions that, among singular points of a generic IAMS, distinguish cuspidal edges from swallowtails. As we shall see in next sections, some of these geometric conditions have natural discrete counterparts and so they will allow us to distinguish the singular vertices of a DIAMS.

Fix a singular point $(u_0,v_0)\in U$ and assume that $\alpha'(u_0)\neq 0$. We may write the singular curve in a neighborhood of 
$(u_0,v_0)$ as $v=v(u)$, $u_0-\epsilon\leq u\leq u_0+\epsilon$. For each $u$, write 
\begin{equation}\label{eq:DefineLambda}
P\alpha'(u)+\lambda P\beta'(u)=0,
\end{equation}
for some $\lambda=\lambda(u)\in\mathbb{R}$. It is clear that $\eta(u)=(1,\lambda(u))$ is in the kernel of $df(u,v(u))$, for any
$u_0-\epsilon\leq u\leq u_0+\epsilon$. Denote by $R:\pi\to\pi$ the reflection with respect to the origin. It is clear that the curves $P\alpha(u)$ and $RP\beta(v(u))$ touches at $u=u_0$.

\begin{prop}\label{prop:SmoothSwallow}
Let $(u_0,v_0)$ be a singular point of a generic IAMS. Then the following conditions are equivalent:

\begin{enumerate}
\item\label{item:Swallow} 
The point $(u_0,v_0)$ is a swallowtail.

\item\label{item:LambdaEqualsVPrime}
$\lambda(u_0)=\tfrac{dv}{du}(u_0)$.

\item\label{item:EqualCurvature} 
The curves $P\alpha(u)$ and $RP\beta(v(u))$ have contact of order $2$ at $u=u_0$.

\item\label{item:Crossing} 
The curves $P\alpha(u)$ and $RP\beta(v(u))$ crosses at $u=u_0$.

\end{enumerate}

\end{prop}

\begin{proof}
Since $(1,\tfrac{dv}{du})$ is tangent to the singular curve, we conclude that $\eta$ is tangent to the singular curve if and only if $\lambda=\tfrac{dv}{du}$. Thus the equivalence between itens \ref{item:Swallow} and \ref{item:LambdaEqualsVPrime} follows from Proposition \ref{prop:SmoothSingularities}.
Differentiate Equation \eqref{eq:DefineLambda} with respect to $u$ and take $u=u_0$ to obtain 
$$
[P\alpha'(u_0),P\alpha''(u_0)]=\lambda^2(u_0)[P\beta'(v_0),P\beta''(v_0)]\tfrac{dv}{du}(u_0).
$$
Thus $\lambda(u_0)=\tfrac{dv}{du}(u_0)$ if and only if
$$
[P\alpha'(u_0),P\alpha''(u_0)]=\lambda^3(u_0)[P\beta'(v_0),P\beta''(v_0)].
$$
We may assume w.l.o.g. that $\lambda(u_0)>0$. In this case, $P\alpha'(u_0)$ and $P\beta'(v_0)$ are pointing to opposite directions and the above equation becomes
$$
\frac{[P\alpha'(u_0),P\alpha''(u_0)]}{|P\alpha'(u_0)|^3}=\frac{[P\beta'(v_0),P\beta''(v_0)]}{|P\alpha'(u_0)|^3}.
$$
This is equivalent to say that $P\alpha$ and $RP\beta$ have the same orientation and the same signed curvature at $(u_0,v_0)$. We conclude that itens \ref{item:LambdaEqualsVPrime} and \ref{item:EqualCurvature} are equivalent. 
The equivalence between itens \ref{item:EqualCurvature} and \ref{item:Crossing} is obvious.
\end{proof}

\section{Discrete indefinite affine minimal surfaces}

Given a function $g:U\subset\mathbb{Z}^2\to\mathbb{R}^k$, we denote its discrete derivatives by
$$
g_u(u+\tfrac{1}{2},v)=g(u+1,v)-g(u,v),\ \ \ g_v(u,v+\tfrac{1}{2})=g(u,v+1)-g(u,v).
$$
The second mixed discrete derivative is then given by
$$
g_{uv}(u+\tfrac{1}{2},v+\tfrac{1}{2})=g(u+1,v+1)+g(u,v)-g(u+1,v)-g(u,v+1).
$$

\subsection{Asymptotic nets, Lelieuvre normals and affine minimal surfaces}

Given a map $f:U\subset\mathbb{Z}^2\to\mathbb{R}^3$ and $(u,v)\in U$, we call the vertices $f(u\pm 1,v)$, 
$f(u, v\pm 1)$ and $f(u,v)$ together with the edges connecting them the {\it star} of $(u,v)$. 
The map $f$ is called a {\it discrete asymptotic net} if the star of $(u,v)$ is planar, for any $(u,v)\in U$.

Given a discrete asymptotic net $f$, we say that $\nu:U\subset\mathbb{Z}^2\to\mathbb{R}^3$ is a {\it Lelieuvre 
normal vector field} associated to $f$ if $\nu(u,v)$ is normal to the star plane at $(u,v)$ and the discrete Lelieuvre's equations 
\begin{equation*}\label{DiscreteLelieuvre}
f_u(u+\tfrac{1}{2},v)=\nu(u,v)\times \nu_u(u+\tfrac{1}{2},v),\ \ f_v(u,v+\tfrac{1}{2})=-\nu(u,v)\times \nu_v(u,v+\tfrac{1}{2})
\end{equation*}
hold (\cite{Bobenko2008}). 
It is well-known that, given a normal vector $\nu_0$ at a fixed point $(u_0,v_0)\in U$, there exists a unique Lelieuvre normal vector field $\nu:U\subset\mathbb{Z}^2\to\mathbb{R}^3$ such that $\nu(u_0,v_0)=\nu_0$. If we multiply $\nu_0$ by a positive real number $\alpha$, then $\nu(u,v)$ will be multiplied by $\alpha$ or $\alpha^{-1}$, according to the relative parity  
of $u+v$ with respect to $u_0+v_0$ (black-white rescaling).

We say that $\nu:U\subset\mathbb{Z}^2\to\mathbb{R}^3$ is a {\it Moutard net} if 
\begin{equation*}\label{Moutard}
\nu(u+1,v+1)+\nu(u,v)=\rho(u+\tfrac{1}{2},v+\tfrac{1}{2})\left( \nu(u,v+1)+\nu(u+1,v) \right),
\end{equation*}
for a certain scalar function $\rho$.
Any Lelieuvre normal vector field of an asymptotic net is a Moutard net. Conversely, given a Moutard net $\nu$, there exist an asymptotic net, unique up to translations, such that $\nu$ is one of its Lelieuvre normal vector fields (\cite{Bobenko2008}).

An {indefinite discrete affine minimal surface} (DIAMS) is a discrete asymptotic net that admits a Lelieuvre normal vector field $\nu$ satisfying 
$\nu_{uv}=0$, or equivalently $\rho=1$ (\cite{Craizer2010}).
It is clear that for a DIAMS we can write 
\begin{equation}\label{eq:NormalMinimal}
\nu(u,v)=\alpha(u)-\beta(v)
\end{equation}
for certain polygonal lines $\alpha(u)$ and $\beta(v)$, and conversely, given polygonal lines $\alpha(u)$ and $\beta(v)$,
Equation \eqref{eq:NormalMinimal} defines a Lelieuvre normal of a DIAMS.

In this work, we shall need the following generic hypothesis on the polygonal lines $\alpha$ and $\beta$: Given $u\in I$, $v\in J$, consider the line through $\alpha(u)$ and $\beta(v)$ and then the four planes determined by this line and $\alpha(u\pm 1)$, $\beta(v\pm 1)$. We shall assume that, for any pair $(u,v)\in I\times J$, these four planes are distinct.

Define, similarly to the smooth case,
\begin{equation*}
M\left(u+\tfrac{1}{2},v+\tfrac{1}{2}\right)=\left[f_u\left(u+\tfrac{1}{2},v\right),f_v\left(u,v+\tfrac{1}{2}\right),f_{uv}\left(u+\tfrac{1}{2},v+\tfrac{1}{2}\right)\right].
\end{equation*}
One can verify that 
\begin{equation*}
M(u+\tfrac{1}{2},v+\tfrac{1}{2})=\Omega^2(u+\tfrac{1}{2},v+\tfrac{1}{2}),
\end{equation*}
where
$$
\Omega\left(u+\tfrac{1}{2},v+\tfrac{1}{2}\right)=\left[\nu(u,v),\nu_u\left(u+\tfrac{1}{2},v\right),\nu_v\left(u,v+\tfrac{1}{2}\right)\right].
$$
So
\begin{equation}\label{eq:DiscreteOmega}
\Omega(u+\tfrac{1}{2},v+\tfrac{1}{2})=\left[\alpha(u)-\beta(v),\alpha'(u+\tfrac{1}{2}),\beta'(v+\tfrac{1}{2})\right],
\end{equation}
which corresponds to the discrete counterpart of Equation \eqref{eq:SmoothOmega}.

\subsection{Bilinear interpolation}

Consider a discrete asymptotic net. 
For a quadrangle $(u+\tfrac{1}{2},v+\tfrac{1}{2})$, there exists a one parameter family of quadrics 
$Q(u+\tfrac{1}{2},v+\tfrac{1}{2})$ that passes through the edges of the quadrangle. We shall refer to them as
{\it interpolating} quadrics. Two quadrics defined at adjacent quadrilaterals are called {\it compatible} 
if their tangent planes coincide at their common edge. If we assume that $U$ is simply connected, given an interpolating quadric $Q_0$ at $(u_0+\tfrac{1}{2},v_0+\tfrac{1}{2})$, there exists a unique collection of compatible interpolating quadrics $Q(u+\tfrac{1}{2},v+\tfrac{1}{2})$, $(u,v)\in U$, such that $Q(u_0+\tfrac{1}{2},v_0+\tfrac{1}{2})=Q_0$ (\cite{Rorig2014}).
We shall refer to such a collection as a {\it field of compatible interpolating quadrics}. Observe that such a field depends on the choice of a initial quadric $Q_0$. 

%The results of the above paragraph can be rephrased in terms of Lelieuvre normals: The co-normal vector field of an interpolating quadric at the vertices of the quadrangle determine Lelieuvre normals, and two adjacent interpolating quadrics are compatible if and only if their co-normals coincide at the common vertices. Thus,  given a normal vector $N_0$ at $(u_0,v_0)$, there exists a unique field of compatible interpolating quadrics $Q(u+\tfrac{1}{2},v+\tfrac{1}{2})$, $(u,v)\in U$, such that the co-normal of the quadric $Q(u_0\pm\tfrac{1}{2},v_0\pm\tfrac{1}{2})$ at $(u_0,v_0)$ is $N_0$.

In the case of DIAMS, there exists a natural choice %for the Lelieuvre normals, and hence a natural choice 
for the field of interpolating quadrics. In fact, one can verify that a discrete asymptotic net is a DIAMS if and only if there exists a field of compatible interpolating quadrics consisting of hyperbolic paraboloids (\cite{Craizer2010}). In this case, the paraboloids are obtained by bilinear interpolation of the vertices.
%, i.e., if the vertices of a quadrangle are $A,B,C,D$, the hyperbolic paraboloid is parameterized by 
%\begin{equation}\label{BilinearInterpolator}
%A+s(B-A)+t(C-A)+st(D+A-B-C),\ \ 0\leq s\leq 1, \ \ 0\leq t\leq 1.
%\end{equation}
We shall use these interpolations to provide figures of the discrete asymptotic net that look like $C^1$ smooth surfaces. 
%As we shall see, these surfaces resembles smooth cuspidal edges at discrete cuspidal edges vertices and smooth swallowtails at discrete swallowtail vertices. 

\section{Singularities of a DIAMS}

\subsection{The singular set as a polygonal line}

Consider an edge $(u_0,v_0+\tfrac{1}{2})$ of a DIAMS. We say that this edge is a {\it singular edge} if 
\begin{equation}\label{eq:SingularvEdge}
\Omega(u_0-\tfrac{1}{2},v_0+\tfrac{1}{2})\cdot \Omega(u_0+\tfrac{1}{2},v_0+\tfrac{1}{2})<0.
\end{equation}
Similarly, an edge $(u_0+\tfrac{1}{2},v_0)$ is a {\it singular edge} if 
\begin{equation}\label{eq:SingularuEdge}
\Omega(u_0+\tfrac{1}{2},v_0-\tfrac{1}{2})\cdot \Omega(u_0+\tfrac{1}{2},v_0+\tfrac{1}{2})<0.
\end{equation}
A vertex which is adjacent to 
a singular edge will be called a {\it singular vertex}. The singular set is the union of singular edges and vertices.

Fix a singular vertex $(u_0,v_0)$ and, as in the smooth case, consider a plane $\pi=\pi(u_0,v_0)$ transversal to $\alpha(u_0)-\beta(v_0)$. Denote by $P$ the projection on $\pi$ parallel to $\alpha(u_0)-\beta(v_0)$.

\begin{prop}
An edge $(u_0,v_0+\tfrac{1}{2})$ is singular if and only if the sides $P\alpha'(u_0\pm\tfrac{1}{2})$ are in the same half-plane of $\pi$ determined by the line $tP\beta'(v_0+\tfrac{1}{2})$, $t\in\mathbb{R}$. Similar conditions hold for the edges
$(u_0,v_0-\tfrac{1}{2})$ and $(u_0\pm\tfrac{1}{2},v_0)$. 
\end{prop}

\begin{proof}
From Equation \eqref{eq:DiscreteOmega}, we have that
$$
\Omega(u_0\pm\tfrac{1}{2},v_0+\tfrac{1}{2})=\left[ \alpha(u_0)-\beta(v_0), P\alpha'(u_0\pm\tfrac{1}{2}),P\beta'(v_0+\tfrac{1}{2})  \right].
$$
Thus the edge $(u_0,v_0+\tfrac{1}{2})$ is singular if and only if $P\alpha'(u_0+\tfrac{1}{2})$ and $P\alpha'(u_0-\tfrac{1}{2})$
determine, together with $P\beta'(v_0+\tfrac{1}{2})$, different orientations for the plane $\pi$.
\end{proof}

We need to make one more assumption before proceeding with our analysis of singular vertices:
Choose $\alpha(u_0)$, $\beta(v_0)$, $\alpha(u_0\pm 1)$ and $\beta(v_0+1)$  such that the edge $(u_0,v_0+\tfrac{1}{2})$ is singular. In the plane $\pi$, consider the lines through the origin parallel to $P\alpha'(u_0\pm\tfrac{1}{2})$, which divide $\pi$ in four quadrants. We shall assume that $P\beta(v_0-1)$ is not in the quadrant containing $P\beta(v_0+\tfrac{1}{2})$ (see Figure \ref{Fig:DIAMSForbidden}). 
This hypothesis is quite natural if we think that $\alpha$ and $\beta$ are discretizations of smooth regular curves.

\begin{figure}[!htb]
\includegraphics[width=.5\linewidth]{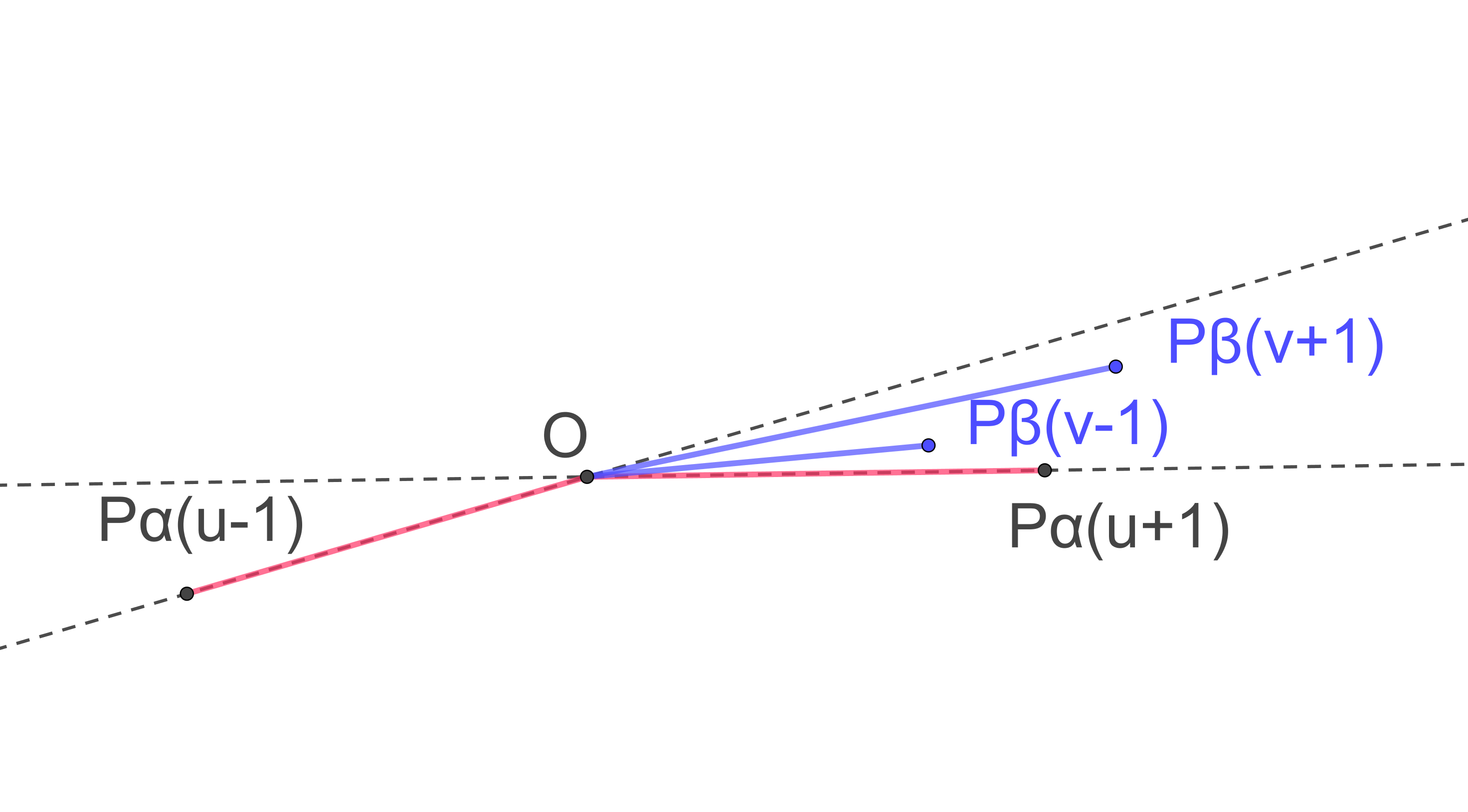}
\caption{\small Forbidden configuration of the projections of $\alpha$ and $\beta$ in $\pi$, assuming that $(u_0,v_0+\tfrac{1}{2})$ is a singular edge.}\label{Fig:DIAMSForbidden}
\end{figure}

The lines through the origin parallel to $P\alpha'(u_0\pm\tfrac{1}{2})$ divide the plane $\pi$ in four quadrants, and the one containing $P\beta(v+1)$ is forbidden for $P\beta(v-1)$. We shall refer to the quadrant opposite to the one containing $P\beta(v+1)$ as quadrant A. We call quadrant B the one such that, if $P\beta(v-1)$ is positioned in it, the polygonal lines $P\beta$ and $P\alpha$ cross at the origin. The other quadrant will called C: if $P\beta(v-1)$ is positioned in quadrant C, the polygonal lines $P\beta$ and $P\alpha$ do not cross at the origin. If we start with a different singular edge, namely $(u_0,v_0-\tfrac{1}{2})$ or $(u_0\pm\tfrac{1}{2},v_0)$, the labeling of the quadrants would be similar. 
The possible positions of $P\beta(v_0-1)$ are shown in Figure \ref{Fig:DIAMSConfiguration}. In the up-left figure, $P\beta(v-1)$ is positioned in quadrant A, in the up-right figure, $P\beta(v-1)$ 
is positioned in quadrant B, while in the down figure, $P\beta(v-1)$ is positioned in quadrant C. 

\begin{figure}[!htb]
\includegraphics[width=.4\linewidth]{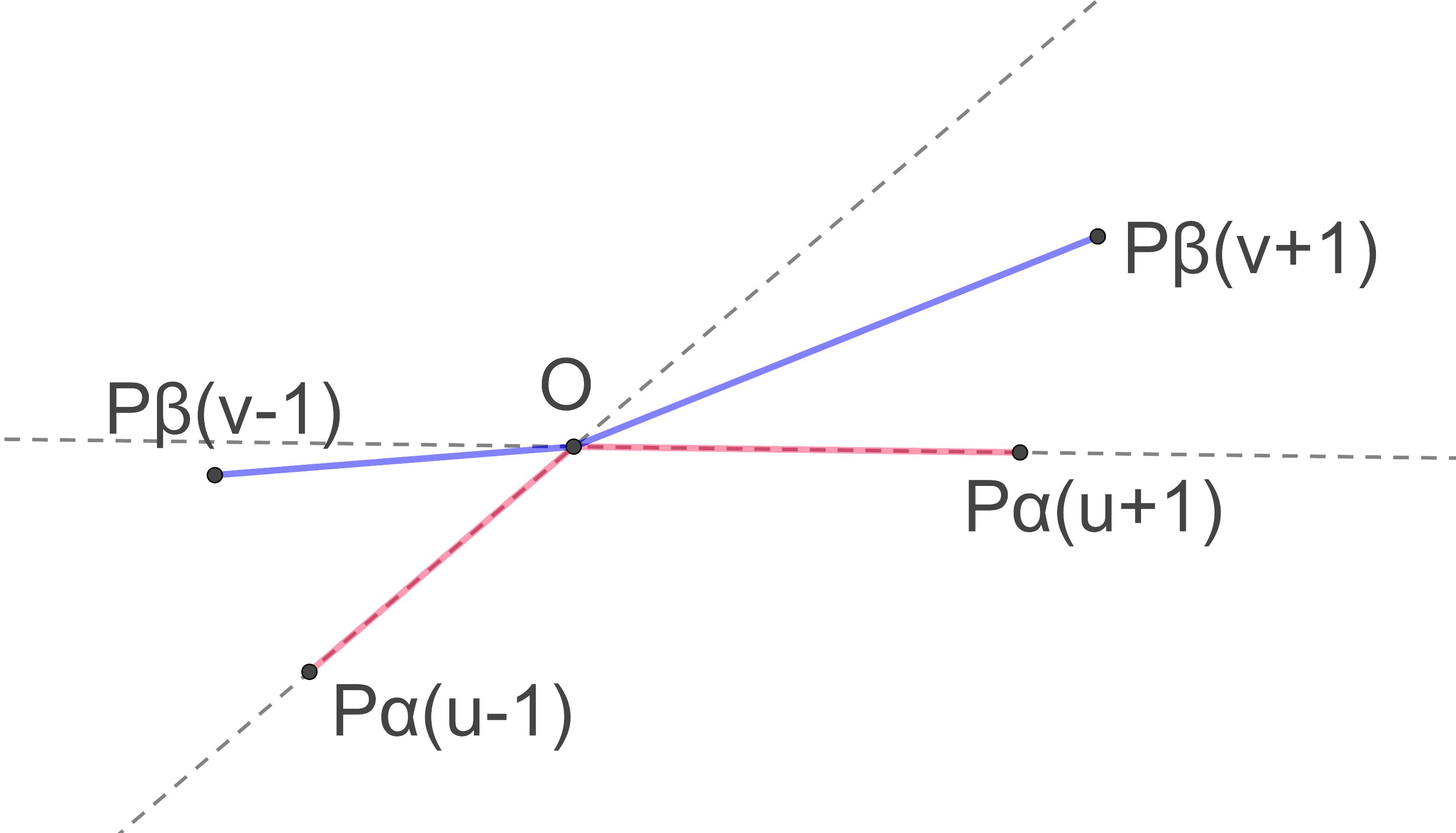}
\hfill
\includegraphics[width=.4\linewidth]{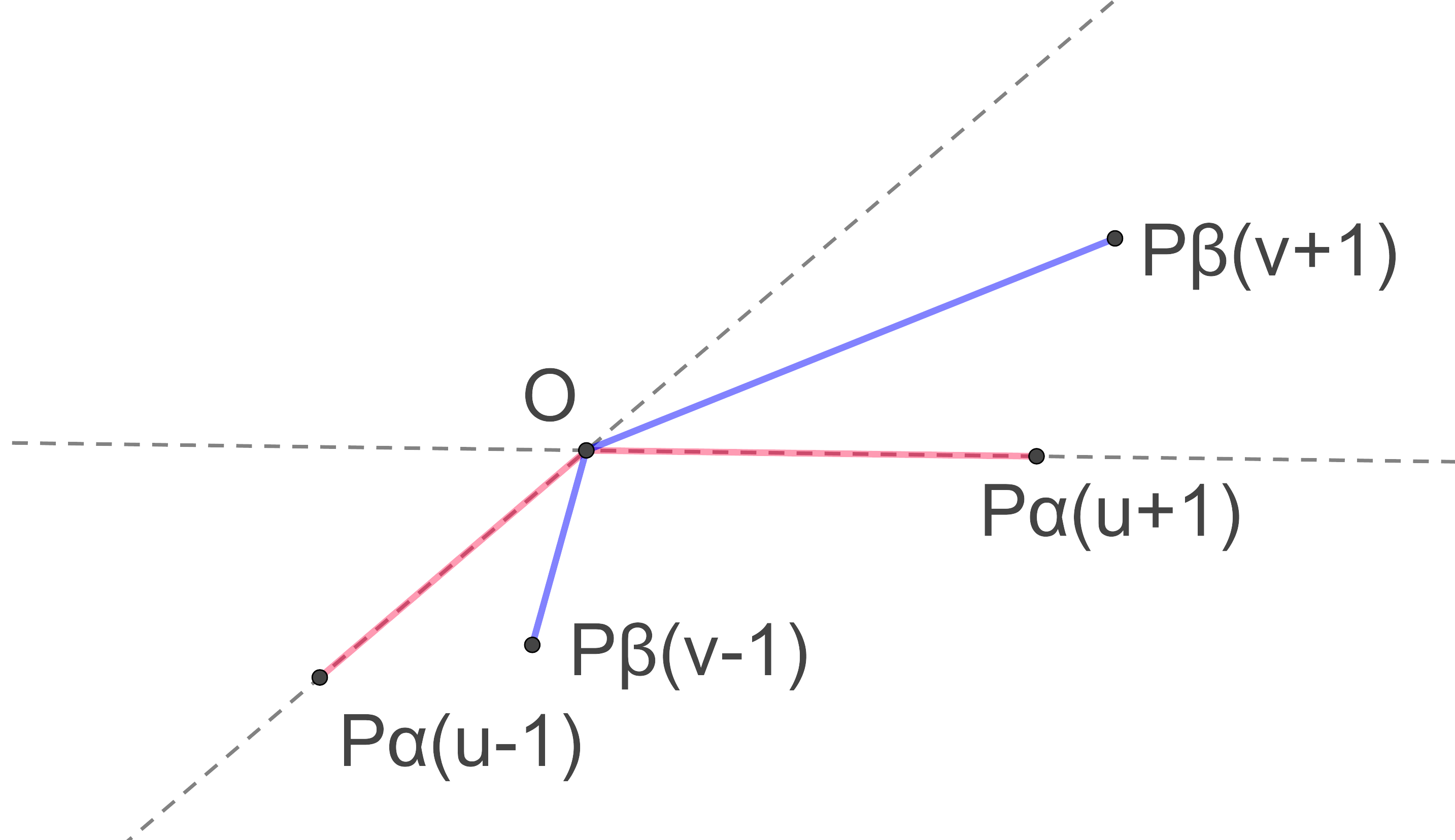}
\hfill
\includegraphics[width=.4\linewidth]{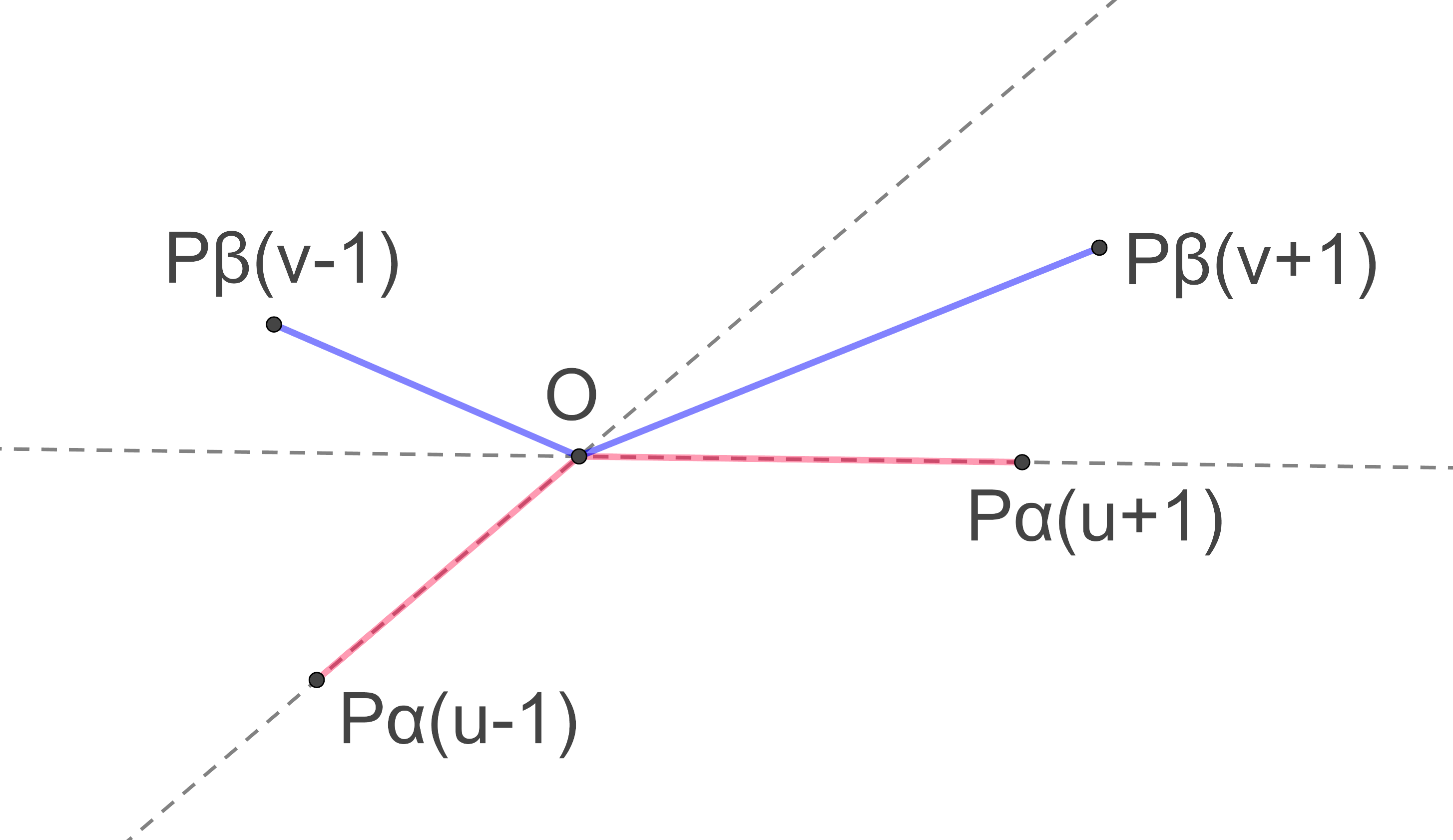}
\caption{\small Possible configurations of the projections of $\alpha$ and $\beta$ in $\pi$, assuming that $(u_0,v_0+\tfrac{1}{2})$ is a singular edge. In the up-left, up-right and down figures, $P\beta(v-1)$ is positioned in quadrant A, B or C, respectively. }\label{Fig:DIAMSConfiguration}
\end{figure}

Next proposition is a discrete counterpart of the fact that, in the smooth generic case, the singular set is a smooth curve:

\begin{prop}
In a DIAMS, the singular set is a polygonal line, i.e., each singular vertex is adjacent to exactly two singular edges.  
\end{prop}

\begin{proof}
Observe that in configuration A, only $(u,v+\tfrac{1}{2})$ and $(u,v-\tfrac{1}{2})$ are singular edges, in configuration B, only $(u,v+\tfrac{1}{2})$ and $(u-\tfrac{1}{2},v)$ are singular edges, while in configuration C, only $(u,v+\tfrac{1}{2})$ and $(u+\tfrac{1}{2},v)$ are singular edges, which proves that the singular set is a polygonal line.
\end{proof}

%\subsection{Star configurations}

Using Lelieuvre's formulas, one can obtain the star configuration at $(u,v)$ from the configurations of the projections of 
$\alpha(u\pm\tfrac{1}{2})$ and $\beta(v\pm\tfrac{1}{2})$ in the plane $\pi$. 
The A, B, C cases of Figure \ref{Fig:DIAMSConfiguration} correspond to the star configurations 
of Figure \ref{Fig:StarConfiguration}, left, middle, right, respectively.

%We shall classify a singular vertex $(u,v)$ according to the configuration of its star. At the $(u,v)$-plane, we consider $f(u,v)$ as the origin and denote $f(u+1,v)=e_1$, $f(u,v+1)=e_2$. We shall assume 
%that the edge $(u, v+\tfrac{1}{2})$ is singular. This implies that $f(u-1,v)=ae_1+be_2$, with $a>0$. We may also assume that $b>0$, for if not, we interchange the r\^oles of $f(u-1,v)$ and $f(u+1,v)$.

%The configuration of the star is then determined by the position of $f(u,v-1)$. There are three possible relative positions, as shown in Figure \ref{Fig:StarConfiguration}. 

\begin{figure}[!htb]
\includegraphics[width=.3\linewidth]{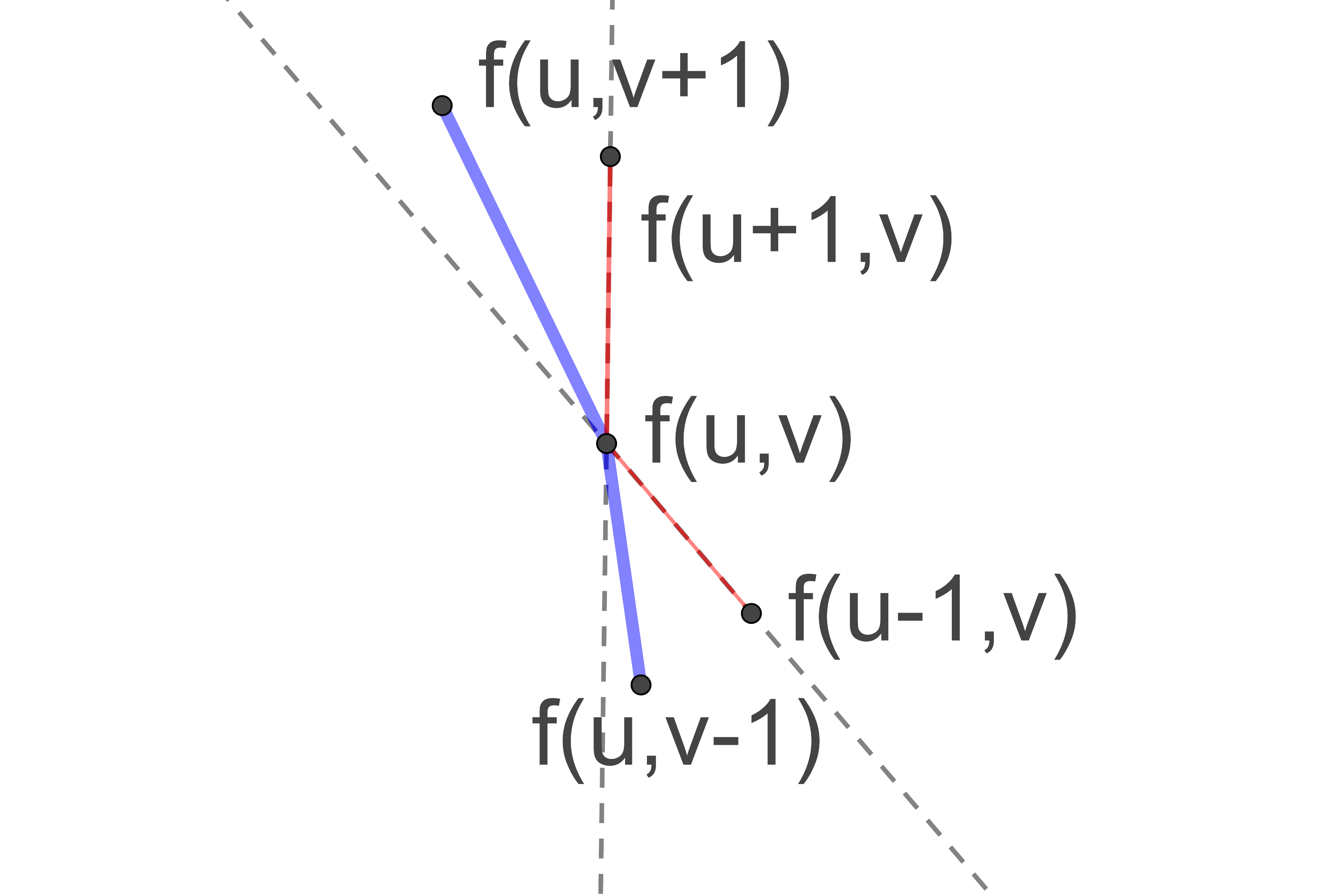}
\hfill
\includegraphics[width=.3\linewidth]{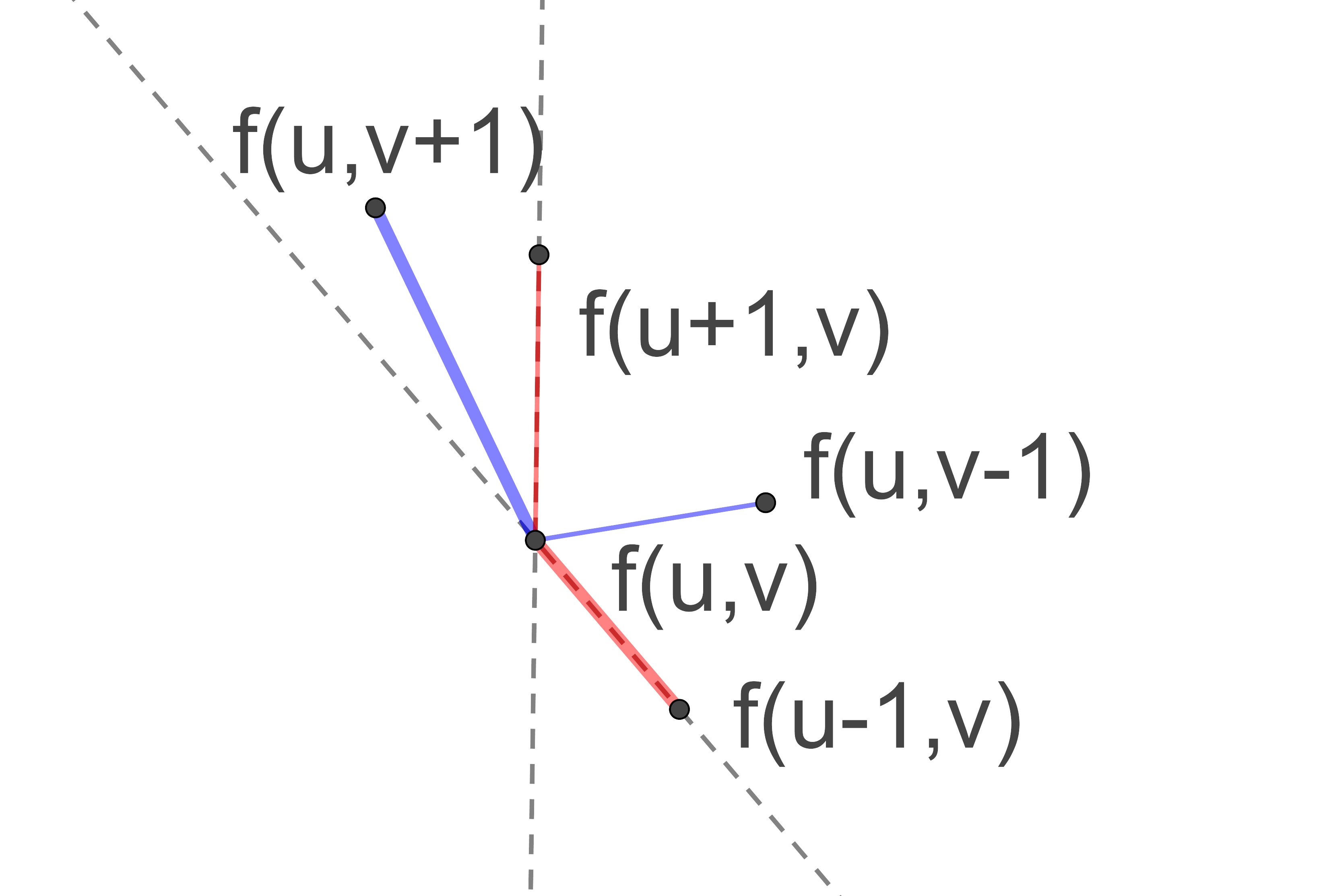}
\hfill
\includegraphics[width=.3\linewidth]{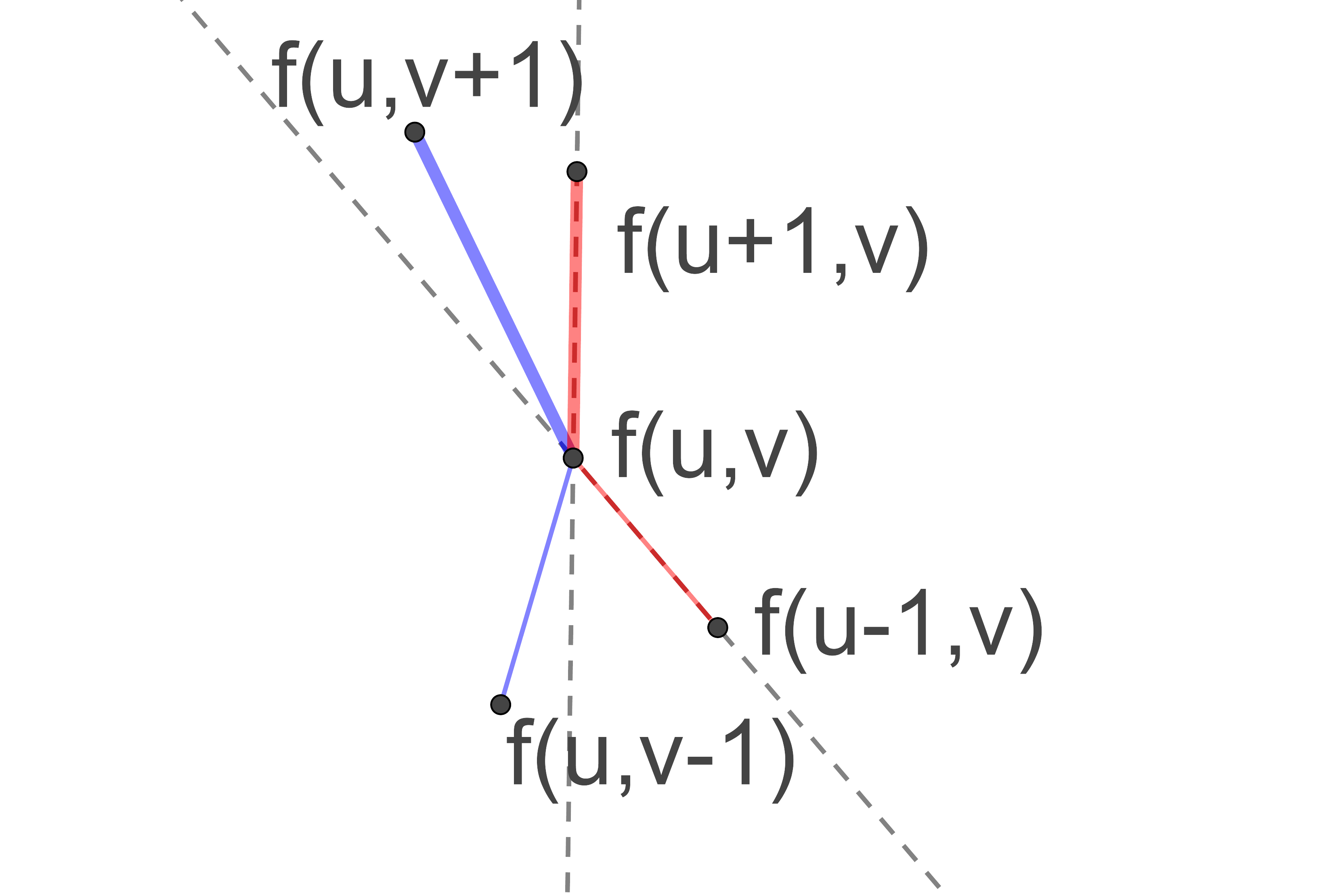}
\caption{\small Star configurations of cases A (left), B (middle) and C (right). Thick full segments correspond to singular edges.}\label{Fig:StarConfiguration}
\end{figure}

\subsection{Types of the singular vertices}

In this section we distinguish between vertices that we should consider as cuspidal from those that we should consider swallowtails. For this, we shall propose discrete counterparts of some properties of the swallowtail points of a smooth IAMS. 

\begin{enumerate}

\item
In the smooth case, a characteristic property of a swallowtail is that $P\alpha$ and $RP\beta$ crosses at $(u_0,v_0)$ (Proposition \ref{prop:SmoothSwallow} (d)), where R denotes the reflection in the plane $\pi$ with respect to the origin. Among configurations A, B and C of Figure \ref{Fig:DIAMSConfiguration}, only C present this property. 

\item
Another characteristic property of a smooth swallowtail is that, assuming $P\alpha'(u_0)$ and $P\beta'(v_0)$ are pointing to the same direction, the function $v(u)$, or $u(v)$, along the singular curve is strictly decreasing (Proposition \ref{prop:SmoothSwallow} (b) with $\lambda<0$). The discrete counterpart of this property is that, along the singular curve, the $u$ and $v$ edges are traversed in opposite directions when passing through $(u_0,v_0)$. Among  configurations A, B and C of Figure \ref{Fig:DIAMSConfiguration}, only C present this property.

\item
In the smooth case, the singular curve at the surface is regular in the case of a cuspidal edge and has a cusp in the case of a swallowtail. Moreover, a planar curve traverse a transversal curve at a regular point but remains in the same side at a cusp point. 
In the discrete case, the only transversal curves that we have at hand are the $u$ and $v$ curves. The discrete counterpart of this property is that both singular edges are in the same side of the lines determined by the other edges. Among the star configurations A, B and C of Figure \ref{Fig:StarConfiguration}, only C present this property. 

\end{enumerate}

We have thus proved the following proposition:

\begin{prop}\label{prop:DiscreteSwallow}
Let $(u_0,v_0)$ be a singular vertex of a DIAMS. Denote by $P$ the projection in a plane $\pi$ transversal to $\alpha(u_0)-\beta(v_0)$ parallel to this vector, and by $R:\pi\to\pi$ the reflection with respect to the origin. The following statements are equivalent:
\begin{enumerate}
\item
$P\alpha$ and $RP\beta$ crosses at $(u_0,v_0)$.

\item
Along the singular curve, the $u$ and $v$ edges are traversed in opposite directions when passing through $(u_0,v_0)$.

\item\label{prop3}
Both singular edges are in the same side of the lines determined by the other edges.
\end{enumerate}
\end{prop}

From the above proposition, the distinction between ``discrete swallowtails'' and ``discrete cuspidal edges'' becomes clear.
The following definition is a generalization of the case of improper affine spheres (\cite{Craizer2023}):

\begin{defn}
Consider a DIAMS. A singular vertex $(u_0,v_0)$ is called a {\it discrete swallowtail} if it satisfies one, and hence all, of the properties of Proposition \ref{prop:DiscreteSwallow}. If this is not the case, the vertex is called a {\it discrete cuspidal edge}.
\end{defn}

\begin{rem}
The definition of singular edges and Proposition \ref{prop:DiscreteSwallow}(\ref{prop3}) refer only to asymptotic net, not to the DIAMS itself. So the definitions of this section may be used for other types of discrete asymptotic nets not coming from a DIAMS. 
\end{rem}

In next example we consider DIAMS with four quadrangles. The bilinear interpolation is used to visually reinforce
the difference between a singular vertex that is a cuspidal edge from a singular vertex that is a swallowtail.

\begin{example}\label{ex:1}
Consider the polygonal lines $\alpha,\beta:\{-1,0,1\}\to\mathbb{R}^3$ given by
$$
\alpha(-1)=\alpha(0)+(-1,-1,1)\Delta u,\ \ \alpha(0)=(0,0,1),\ \ \alpha(1)=\alpha(0)+(1,0,0)\Delta u,
$$
$$
\beta(-1)=(-1,y,0)\Delta v,\ \ \beta(0)=(0,0,0),\ \ \beta(1)=(2,1,0)\Delta v.
$$
with $\Delta u=\Delta v=0.1$. We consider three different values for $y$: (A) $y=-1/2$, \  (B) $y=-2$, \ (C) $y=1$. 
In Figure \ref{Fig:Examples}, one can see the corresponding DIAMS with a bilinear interpolation in each quadrangle. 
\end{example}

\begin{figure}[!htb]
\includegraphics[width=.3\linewidth]{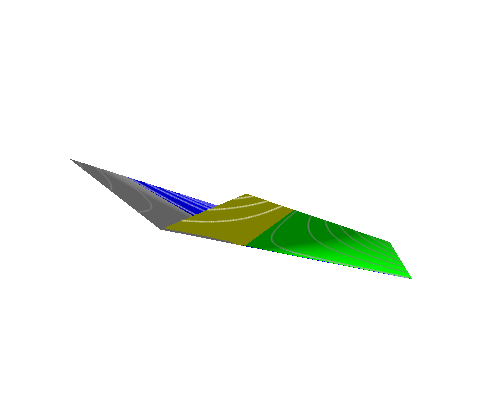}
\hfill
\includegraphics[width=.3\linewidth]{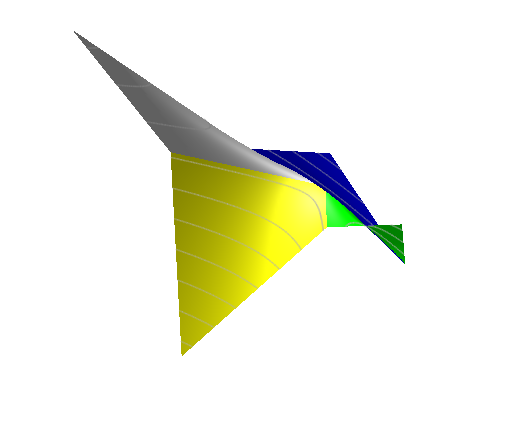}
\hfill
\includegraphics[width=.3\linewidth]{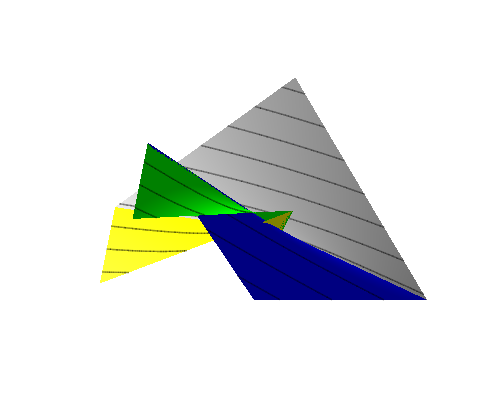}
\caption{\small Example \ref{ex:1} with three different values for $y$, corresponding to configurations A, B and C. The different colors correspond to the bilinear interpolations at the four quadrangles. }\label{Fig:Examples}
\end{figure}

% ----------------------------------------------------------------
\bibliographystyle{amsplain}

\end{document}